\documentclass[submission]{dmtcs-episciences}
\usepackage{amsmath,amsthm,latexsym,amssymb,pgf,url} 

\usepackage{tikz} 
\usetikzlibrary{arrows,decorations.pathmorphing,backgrounds,positioning,fit,petri,calc}

\usepackage[section]{placeins}
\usepackage{fancyhdr,lipsum}

\newtheorem{defi}{Definition}
\newtheorem{theo}{Theorem}
\newtheorem{coro}{Corollary}
\newtheorem{prop}{Proposition}

\newtheorem{la}{Lemma}

\newcommand{\old}[1]{{}}

\author{\'Eva Czabarka\affiliationmark{1,2}
\and {Peter Dankelmann\affiliationmark{2}\thanks{National Research Foundation of South Africa, grant number 118521}}
\and {Trevor Olsen\affiliationmark{1}\thanks{NSF DMS, grant number  1600811}}
\and {L\'aszl\'o A. Sz\'ekely\affiliationmark{1,2}$^\dagger$}
}

\title{Wiener Index and Remoteness in Triangulations and Quadrangulations} 

\affiliation{
University of South Carolina, USA \\
University of Johannesburg, South Africa
}


\received{2020-05-13}
\revised{2020-10-29}
\accepted{2021-01-20}

\begin{document}

\publicationdetails{23}{2021}{1}{3}{6473}

\maketitle

\begin{abstract}
Let $G$ be a connected graph. 
The Wiener index of a connected graph is the sum of the distances between
all unordered pairs of vertices. We provide asymptotic formulas for the maximum
Wiener index of simple triangulations and quadrangulations with given connectivity,
as the order increases, and make conjectures for the extremal triangulations
and quadrangulations based on computational evidence.
If $\overline{\sigma}(v)$ denotes the arithmetic mean of the distances from $v$ to all other vertices of
$G$, then the %
remoteness of $G$ is defined as the 
 largest value of $\overline{\sigma}(v)$ over all vertices $v$ of $G$. 
We give 
sharp upper
bounds on the remoteness of simple triangulations and quadrangulations of given order
and connectivity. 
\end{abstract}

\section{Definitions and Selected Results on the Wiener Index}

Let $G$ be a connected graph. The Wiener index $W(G)$ of $G$ is the 
sum of the distances between all unordered pairs of distinct vertices, i.\ e.,
\[ W(G) = \sum_{\{u,v\} \subseteq V(G)} d_G(u,v), \]
where $d_G(u,v)$ is the usual distance between vertices $u$ and $v$, i.e., the 
minimum number of edges on a $(u,v)$-path in $G$.  
The Wiener index was first studied by the chemist Wiener \cite{Wie1947}, who
observed that it relates well to the boiling point of certain alkanes. 
Several other applications in chemistry were found subsequently, see for
example \cite{Rou2002}.

The systematic study of the mathematical properties of the Wiener index began with
the classical papers by Doyle and Graver \cite{DoyGra1977}, Entringer, Jackson and 
Snyder \cite{EntJacSny1976} and Plesn\'\i k \cite{Ple1984}. 
Several bounds on the Wiener index and closely related parameters, such as transmission 
or routing cost (defined as the sum of the distances between all ordered pairs of vertices),  
average distance or  mean distance  (both are defined as the arithmetic mean of the distances 
between all unordered pairs of distinct vertices) have been proved since.

The most basic upper bound on $W(G)$ states that if $G$ is a connected graph of 
order $n$, then
\begin{equation} \label{maxwiener}
 W(G) \leq \frac{(n-1)n(n+1)}{6}, 
 \end{equation}
which is attained only by paths (see \cite{DoyGra1977} \cite{Ple1984}, \cite{Lov1979}). 
Many sharp or asymptotically sharp bounds on $W(G)$ in terms of other graph parameters are known,
for example minimum degree (\cite{BeeRieSmi2001}, \cite{DanEnt2000}, \cite{KouWin1997}), 
connectivity (\cite{DanMukSwa2009}, \cite{FavKouMah1989}), edge-connectivity 
(\cite{DanMukSwa2008}, \cite{DanMukSwa2008-2}) and maximum degree \cite{FisHofRauSzeVol2002}. 
For recent results on the Wiener index see, for example, \cite{DasNad2017},
\cite{GutCruRad2014}, \cite{KlaNad2014}, \cite{KnoLuzSkrGut2014}, \cite{KnoSkrTep2016} 
\cite{LiMaoGut2018}, \cite{MukVet2014}, \cite{WagWan2009}, \cite{Wag2006} and 
\cite{WanYu2006}.

Entringer et al. \cite{EntJacSny1976}  found that among trees on the same number of vertices, the star minimizes and the
path maximizes the Wiener index (see also \cite{Lov1979} problem 6.23). Fischermann, Hoffmann, Rautenbach, Sz\'ekely and Volkmann
 \cite{FisHofRauSzeVol2002}  (see also \cite{sorder}) characterized
binary trees with minimum and maximum Wiener index. 

The notation we use in this paper is as follows. If $G$ is a graph, then we denote
its vertex set and edge set by $V(G)$ and $E(G)$. By $n(G)$ and $m(G)$ we mean
the  \emph{order} and \emph{size} of $G$, defined as $|V(G)|$ and $|E(G)|$, respectively.
The \emph{eccentricity} $e(v)$ of a vertex $v$ is the distance to a vertex farthest 
from $v$, i.e., $e(v) = \max_{u\in V(G)} d_G(v,u)$. The largest and the smallest 
of the eccentricities of the vertices of $G$ are the \emph{diameter} and the \emph{radius}
of $G$, respectively.  
The \emph{neighborhood} of a vertex
$v$ of $G$ is the set of vertices adjacent to $v$, it is denoted by $N_G(v)$, and
the cardinality $|N_G(v)|$ is the \emph{degree} of $v$, which we denote by ${\rm deg}_G(v)$.   
If $i$ is an integer with $0 \leq i \leq e(v)$, then $N_i(v)$ denotes the set of all
vertices at distance exactly $i$ from $v$, and we write $n_i(v)$ for $|N_i(v)|$. 
If there is no danger of confusion, we often omit the subscript $G$ or the argument
$G$ or $v$.  If $A,B \subseteq V(G)$, then $m(A,B)$ denotes the number of 
edges that join a vertex in $A$ to a vertex in $B$, and $G[A]$ denotes 
the subgraph of $G$ induced by $A$. 
If $w$ is a vertex of $G$ and $A\subseteq V(G)$, then a $(w,A)$-\emph{fan}
is a set of $|A|$ paths from $w$ to $A$, where any two paths have only $w$ in common.
If $G$ is connected and not complete, then
the \emph{connectivity} of $G$, $\kappa(G)$, is the smallest number of vertices whose deletion renders $G$ disconnected. 
The (not necessarily simple) plane graph $G$ is a \emph{triangulation} (resp. \emph{quadrangulation}) if every
face is a triangle (resp. $4$-cycle). A \emph{simple triangulation} (resp. \emph{simple quadrangulation}) is a triangulation (resp. quadrangulation) whose
underlying graph is simple. i.e. has no multiple edges. The graph $H$ is a \emph{minor} of the graph $G$, if $H$ can be obtained from a subgraph of $G$ by
edge-contractions.

By $C_n$, $K_n$ and $\overline{K_n}$ we mean the cycle, the complete graph, 
and the edgeless graph on $n$ vertices, respectively. If $G$ and $H$ are graphs
then $G+H$ denotes the graph obtained from the union of $G$ and $H$ by adding edges 
joining every vertex of $G$ to every vertex of $H$.

\section{Summary of the Results in this Paper}
Another natural class of study for extremal Wiener index
is planar graphs. However, as the maximum Wiener index of graphs (\ref{maxwiener}) is attained by a path, it makes sense to consider
more restricted classes of planar graphs, like simple triangulations and quadrangulations.  Che and Collins \cite{CheCol2019}, and the authors of the present paper \cite{talk}
investigated independently the maximum Wiener index of triangulations and  presented the same   simple triangulation 
of order $n$   (see Figures~\ref{fig:t0},~\ref{fig:t1},~\ref{fig:t2}) with Wiener index
\begin{equation} \label{withcases}
W(T_n) =\frac{1}{3}{n+2\choose 3}  - \frac{1}{3}  \Bigl\lfloor \frac{n+2}{3}\Bigl\rfloor=
\begin{cases}
\frac{n^3}{18} + \frac{n^2}{6} & \text{if }n=3k \\
\frac{n^3}{18} + \frac{n^2}{6} - \frac{2}{9} & \text{if } n=3k+1 \\
\frac{n^3}{18} + \frac{n^2}{6} - \frac{1}{9} & \text{if } n=3k+2 ,\\
\end{cases}
\end{equation}
which they  conjectured  to be optimal. (Note that this sequence is present in the On-Line Encyclopedia of Integer Sequences 
\cite{sloane} under A014125, which is the bisection of A001400.  The displayed closed form  is due to Bruno Berseli  \cite{sloane}.)
 We \cite{talk} announced that this conjecture is asymptotically true before the paper \cite{CheCol2019} was submitted. 
Che and Collins \cite{CheCol2019} verified this conjecture for simple triangulations  of order 
not exceeding $10$.  Using computer, we verified this conjecture for  simple triangulations  of order 
not exceeding $18$, see Table~1 in \cite{ourarxive}, an earlier version of this paper. Very recently, Debarun Ghosh, Ervin Gy\H ori, Addisu Paulos, Nika Salia, Oscar Zamora   
verified this conjecture   \cite{ervinarxive}.

In this paper we prove a {\it generalization} of this conjecture {\it asymptotically}. 
Note that every simple triangulation is 3-connected,   but  a simple triangulation cannot be 6-connected because of the number of edges.  Our main theorem (Theorem~\ref{theo:kpa-conn-tri})   
proves that for any $3\leq \kappa \leq 5$, the Wiener index of any $\kappa$-connected simple triangulation of
order $n$ is at most $\frac{1}{6\kappa}n^3 + O(n^{5/2})$. We also prove in Theorem~\ref{theo:kappa-conn-quad} that  for any $2\leq \kappa \leq 3$, the Wiener index of any $\kappa$-connected simple quadrangulation of
order $n$ is at most $\frac{1}{6\kappa}n^3 + O(n^{5/2})$.   %

We provide constructions matching the upper bounds of Theorems~\ref{theo:kpa-conn-tri}  and \ref{theo:kappa-conn-quad}  for the maximum Wiener index
of triangulations and quadrangulations of given connectivity. We do more, as we exhibit triangulations and quadrangulations 
following patterns by the residue of the order $n$ modulo $\kappa$,    which we conjecture as realizers of the maximum Wiener index.
Our conjectures are based on extensive computations. We detail next these conjectures.

  We constructed 4-connected simple triangulations with Wiener index 
\begin{equation} \label{withcases4}
W(T_n^4) =
\begin{cases}
\frac{n^3}{24} + \frac{n^2}{4} + \frac{n}{3} - 2 & \text{if } n=4k+2 \\
\frac{n^3}{24} + \frac{n^2}{4} + \frac{5n}{24} - 1 & \text{if } n=4k+3 \\
\frac{n^3}{24} + \frac{n^2}{4} + \frac{n}{3} - 2 & \text{if } n=4k \\
\frac{n^3}{24} + \frac{n^2}{4} + \frac{5n}{24} - \frac{3}{2} & \text{if } n=4k+1, \\
\end{cases}
\end{equation}
see Figures~\ref{fig:t4c2}, \ref{fig:t4c3}, \ref{fig:t4c0}, \ref{fig:t4c1}. This proves that Theorem~\ref{theo:kpa-conn-tri} is also asymptotically tight
for $\kappa=4$. Furthermore, we conjecture that the repetition of the  obvious  pattern in these figures provide the extremal triangulations.
Using computer, we verified this conjecture for  simple triangulations  of order 
not exceeding $22$, see Table~\ref{tab:t4csummary}.

We constructed 5-connected simple triangulations with Wiener index 
\begin{equation}
W(T_n^5) =
\begin{cases} \label{withcases5}
\frac{n^3}{30} + \frac{3n^2}{10} - \frac{23n}{15} + \frac{168}{5} & \text{if } n=5k+2 \\
\frac{n^3}{30} + \frac{3n^2}{10} - \frac{23n}{15} + 31 & \text{if } n=5k+3 \\
\frac{n^3}{30} + \frac{3n^2}{10} - \frac{23n}{15} + \frac{161}{5} & \text{if } n=5k+4 \\
\frac{n^3}{30} + \frac{3n^2}{10} - \frac{23n}{15} + 32 & \text{if } n=5k \\
\frac{n^3}{30} + \frac{3n^2}{10} - \frac{23n}{15} + \frac{156}{5} & \text{if } n=5k+1, \\
\end{cases}
\end{equation}
see Figures~\ref{fig:t5c2}, \ref{fig:t5c3wi}, \ref{fig:t5c4},  \ref{fig:t5c0},  \ref{fig:t5c1}. This proves that Theorem~\ref{theo:kpa-conn-tri} is also asymptotically tight
for $\kappa=5$. Furthermore, we conjecture that the repetition of the  obvious  pattern in these figures provide the extremal triangulations. 
We arrived to these conjectures using computer and also guesswork regarding the pattern. 
Therefore these conjectures  for the 5-connected case are less supported with
computational evidence than other conjectures in this paper,  as we were able to do the computation only up to the order $32$, 
see Table~\ref{tab:t5csummary}.  The issue is that the pattern slowly develops, and  orders following the same pattern differ by 5---therefore we do
not have sufficiently many data points to have a very convincing conjecture.

We are indebted to Paul Kainen, who after hearing about our triangulation results, asked whether we can prove similar
results for simple quadrangulations. Recall that any simple quadrangulation is 2-connected, but no simple quadrangulation is 4-connected.
We conjecture that the maximum Wiener index of a simple quadrangulation of order $n$ is
\begin{equation} \label{quadwiener}
W(Q_n) =
\begin{cases}
\frac{n^3}{12} + \frac{7n}{6} - 2 & \text{if } n=2k \\
\frac{n^3}{12} + \frac{11n}{12} - 1 & \text{if } n=2k+1,\\

\end{cases}
\end{equation}
based on Figures~\ref{fig:q0}, \ref{fig:q1}. Furthermore, we conjecture that the repetition of the  obvious  pattern in these figures provide the extremal quadrangulations. Using computer, we verified this conjecture for  simple quadrangulations  of order 
not exceeding $20$, see Table~\ref{tab:qsummary}.

We conjecture that  the maximum Wiener index of a 3-connected simple quadrangulation of order $n$ is 
\begin{equation} \label{quadwiener2}
W(Q_n^3) =
\begin{cases}
\frac{n^3}{18} + \frac{n^2}{3} - \frac{17n}{6} +\frac{206}{9} & \text{if } n=3k+14 \\
\frac{n^3}{18} + \frac{n^2}{3} - \frac{17n}{6} + 20 & \text{if } n=3k+15 \\
\frac{n^3}{18} + \frac{n^2}{3} - \frac{17n}{6} + \frac{184}{9} & \text{if } n=3k+16 ,\\
\end{cases}
\end{equation}
based on Figures~\ref{fig:q3c2}, \ref{fig:q3c0},  \ref{fig:q3c1}. Furthermore, we conjecture that the repetition of the  obvious  pattern in these figures provide the extremal quadrangulations. Using computer, we verified this conjecture for  simple quadrangulations  of order 
not exceeding $28$, see Table~\ref{tab:q3csummary}.

Section~\ref{comput} contains the conjectures stated so far in the form of drawings for some fixed order, but with emphasis on the general pattern:
the red colored part is the repeated pattern. 
Even more, we conjecture based on computational evidence that
those drawings not only provide the maximum Wiener index, but for sufficiently large $n$ they are unique with this property.

We remark here that the result of \cite{ervinarxive} described by formula (\ref{withcases}) does not hold for {\it non-simple} triangulations. For the construction of
non-simple triangulations  with asymptotically larger Wiener indices, see Figure~\ref{nonsimpleTn}.
In fact, we conjecture that these constructions are optimal for non-simple  triangulations. %
The non-simple  quadrangulation on Figure~\ref{nonsimpleQn} has a larger Wiener index than conjectured best simple
quadrangulation on Figure~\ref{fig:q0}, but difference is not in the leading term.  For the rest of the paper, under the terms triangulation 
and quadrangulation we will always understand simple triangulation 
and quadrangulation.

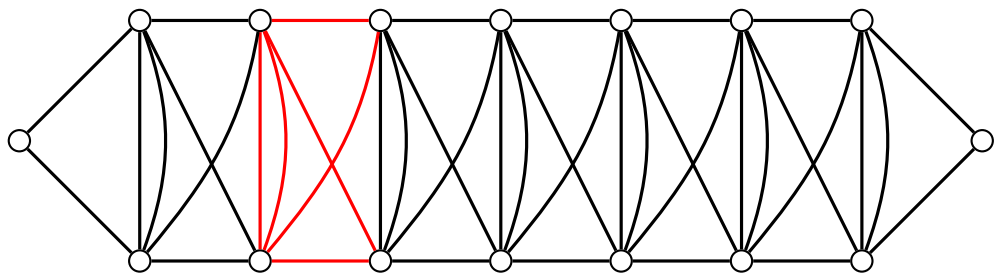
\begin{figure}[htbp]
	\begin{center}
		\begin{tikzpicture}
		[scale=0.8,inner sep=1mm,line/.style={-}%
		vertex/.style={circle,thick,draw}, %
		thickedge/.style={line width=2pt}] %
		\node[circle,thick,draw] (a1) at (0,0) [fill=white] {};
		\node[circle,thick,draw] (a2) at (2,0) [fill=white] {};
		\node[circle,thick,draw] (a3) at (4,0) [fill=white] {};
		\node[circle,thick,draw] (a4) at (6,0) [fill=white] {};
		\node[circle,thick,draw] (a5) at (8,0) [fill=white] {};
		\node[circle,thick,draw] (a6) at (10,0) [fill=white] {};
		\node[circle,thick,draw] (a7) at (12,0) [fill=white] {};
		\node[circle,thick,draw] (x) at (-2,2) [fill=white] {};
		\node[circle,thick,draw] (y) at (14,2) [fill=white] {};
		\node[circle,thick,draw] (c1) at (0,4) [fill=white] {};
		\node[circle,thick,draw] (c2) at (2,4) [fill=white] {};
		\node[circle,thick,draw] (c3) at (4,4) [fill=white] {};
		\node[circle,thick,draw] (c4) at (6,4) [fill=white] {};
		\node[circle,thick,draw] (c5) at (8,4) [fill=white] {};
		\node[circle,thick,draw] (c6) at (10,4) [fill=white] {};
		\node[circle,thick,draw] (c7) at (12,4) [fill=white] {};    
		\draw[very thick] (a1)--(x);
		\draw[very thick] (c1)--(x);
		\draw[very thick] (a7)--(y);
		\draw[very thick] (c7)--(y);
		\draw[very thick] (a1)--(a2) (a3)--(a4)--(a5)--(a6)--(a7);  
		\draw[very thick,red] (a2)--(a3);  

		\draw[very thick] (c1)--(c2) (c3)--(c4)--(c5)--(c6)--(c7);
		 \draw[very thick,red] (c2)--(c3);
		\path[line,very thick,out=-70,in=70] (c1) edge (a1);
		\draw[very thick] (c1)--(a1);  
		\path[red,line,very thick,out=-70,in=70] (c2) edge (a2);
		\draw[very thick,red] (c2)--(a2); 
		\path[line,very thick,out=-70,in=70] (c3) edge (a3);     
		\draw[very thick] (c3)--(a3);  
		\path[line,very thick,out=-70,in=70] (c4) edge (a4);  
		\draw[very thick] (c4)--(a4);  
		\path[line,very thick,out=-70,in=70] (c5) edge (a5); 
		\draw[very thick] (c5)--(a5);  
		\path[line,very thick,out=-70,in=70] (c6) edge (a6); 
		\draw[very thick] (c6)--(a6);  
		\path[line,very thick,out=-70,in=70] (c7) edge (a7);
		\draw[very thick] (c7)--(a7);  
		
		\path[line,very thick,out=50,in=260]  (a1) edge (c2);  
		\path[red,line,very thick,out=50,in=260] (a2) edge (c3); 
		\path[line,very thick,out=50,in=260] (a3) edge (c4); 
		\path[line,very thick,out=50,in=260] (a4) edge (c5); 
		\path[line,very thick,out=50,in=260] (a5) edge (c6); 
		\path[line,very thick,out=50,in=260] (a6) edge (c7); 
		\draw[very thick] (c1)--(a2);  
		\draw[very thick,red] (c2)--(a3); 
		\draw[very thick] (c3)--(a4); 
		\draw[very thick] (c4)--(a5); 
		\draw[very thick] (c5)--(a6); 
		\draw[very thick] (c6)--(a7);    
		\end{tikzpicture}
		\caption{A non-simple  triangulation with larger Wiener index. $W(T_n')=\frac{n^3}{12}+\frac{2n}{3}-1$ for even $n$.}
		\label{nonsimpleTn}
	\end{center}
\end{figure}

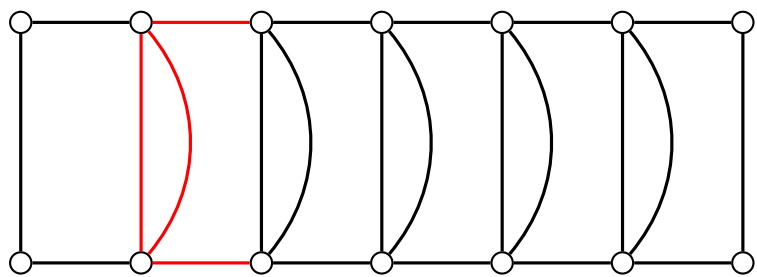
\begin{figure}[htbp]
	\begin{center}
		\begin{tikzpicture}
		[scale=0.8,inner sep=1mm,line/.style={-}%
		vertex/.style={circle,thick,draw}, %
		thickedge/.style={line width=2pt}] %
		\node[circle,thick,draw] (a1) at (0,0) [fill=white] {};
		\node[circle,thick,draw] (a2) at (2,0) [fill=white] {};
		\node[circle,thick,draw] (a3) at (4,0) [fill=white] {};
		\node[circle,thick,draw] (a4) at (6,0) [fill=white] {};
		\node[circle,thick,draw] (a5) at (8,0) [fill=white] {};
		\node[circle,thick,draw] (a6) at (10,0) [fill=white] {};
		\node[circle,thick,draw] (a7) at (12,0) [fill=white] {};
		\node[circle,thick,draw] (c1) at (0,4) [fill=white] {};
		\node[circle,thick,draw] (c2) at (2,4) [fill=white] {};
		\node[circle,thick,draw] (c3) at (4,4) [fill=white] {};
		\node[circle,thick,draw] (c4) at (6,4) [fill=white] {};
		\node[circle,thick,draw] (c5) at (8,4) [fill=white] {};
		\node[circle,thick,draw] (c6) at (10,4) [fill=white] {};
		\node[circle,thick,draw] (c7) at (12,4) [fill=white] {};    
		\draw[very thick] (a1)--(a2) (a3)--(a4)--(a5)--(a6)--(a7);  
		\draw[very thick,red] (a2)--(a3);
		\draw[very thick] (c1)--(c2) (c3)--(c4)--(c5)--(c6)--(c7); 
		\draw[very thick,red] (c2)--(c3);
		\draw[very thick] (c1)--(a1);  
		\path[red,line,very thick,out=-50,in=50] (c2) edge (a2);
		\draw[very thick,red] (c2)--(a2); 
		\path[line,very thick,out=-50,in=50] (c3) edge (a3);     
		\draw[very thick] (c3)--(a3);  
		\path[line,very thick,out=-50,in=50] (c4) edge (a4);  
		\draw[very thick] (c4)--(a4);  
		\path[line,very thick,out=-50,in=50] (c5) edge (a5); 
		\draw[very thick] (c5)--(a5);  
		\path[line,very thick,out=-50,in=50] (c6) edge (a6); 
		\draw[very thick] (c6)--(a6);  
		\draw[very thick] (c7)--(a7);  
		\end{tikzpicture}
		\caption{A non-simple  quadrangulation with larger Wiener index. $W(Q_n')=\frac{n^3}{12}+\frac{n^2}{4}-\frac{n}{3}$ for even $n$.
		For this sequence, 	see A131423 \cite{sloane}.} 
		\label{nonsimpleQn}
	\end{center}
\end{figure}

Che and Collins noted \cite{CheCol2019}  that the {\sl minimum} Wiener index of a triangulation of order $n$ is a trivial problem, as 
Euler's formula determines the number of edges, and there are  constructions, in which every pair of vertices are at most 
distance two.   The situation is analogous for quadrangulations.
For  minimizers, see Figure~\ref{fig:mini}.

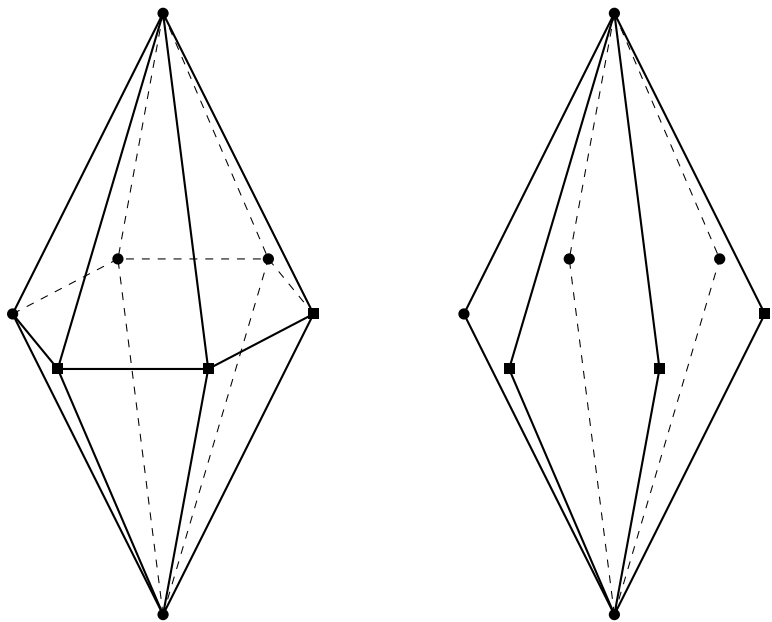
\begin{figure}[htbp]%
\begin{center}
\begin{tikzpicture}[line/.style={-}] 
         \node[fill=black,circle,inner sep=1.5pt]  at (-2,0) {};
        \node[fill=black,circle,inner sep=1.5pt]  at (-.6,0.732) {};
        \node[fill=black,circle,inner sep=1.5pt]  at (1.4,0.732) {};
        \node[fill=black,rectangle,inner sep=2pt]  at (2,0) {};
        \node[fill=black,rectangle,inner sep=2pt]  at (-1.4,-0.732) {};
        \node[fill=black,rectangle,inner sep=2pt]  at (.6,.-0.732) {};
         \node[fill=black,circle,inner sep=1.5pt]  at (0,4) {};
        \node[fill=black,circle,inner sep=1.5pt]  at (0,-4) {};        
 	\draw[thick] (0,4)--(-2,0);
 	\draw[dashed] (0,4)--(-.6,0.732);
 	\draw[dashed] (0,4)--(1.4,0.732);
 	\draw[thick] (0,4)--(2,0);
 	\draw[thick] (0,4)--(-1.4,-0.732);
 	\draw[thick] (0,4)--(.6,-0.732);
 	\draw[thick] (0,-4)--(-2,0);
 	\draw[dashed] (0,-4)--(-.6,0.732);
 	\draw[dashed] (0,-4)--(1.4,0.732);
	\draw[thick] (0,-4)--(2,0);
 	\draw[thick] (0,-4)--(-1.4,-0.732);
 	\draw[thick] (0,-4)--(.6,-0.732);	
 	\draw[thick] (2,0)--(.6,-0.732)--(-1.4,-0.732)--(-2,0);
 	\draw[dashed] (-2,0)--(-.6,0.732)--(1.4,0.732)--(2,0);        
        \node[fill=black,circle,inner sep=1.5pt]  at (4,0) {};
        \node[fill=black,circle,inner sep=1.5pt]  at (5.4,0.732) {};
        \node[fill=black,circle,inner sep=1.5pt]  at (7.4,0.732) {};
        \node[fill=black,rectangle,inner sep=2pt]  at (8,0) {};
        \node[fill=black,rectangle,inner sep=2pt]  at (4.6,-0.732) {};
        \node[fill=black,rectangle,inner sep=2pt]  at (6.6,.-0.732) {};
         \node[fill=black,circle,inner sep=1.5pt]  at (6,4) {};
        \node[fill=black,circle,inner sep=1.5pt]  at (6,-4) {};        
 	\draw[thick] (6,4)--(4,0);
 	\draw[dashed] (6,4)--(5.4,0.732);
 	\draw[dashed] (6,4)--(7.4,0.732);
 	\draw[thick] (6,4)--(8,0);
 	\draw[thick] (6,4)--(4.6,-0.732);
 	\draw[thick] (6,4)--(6.6,-0.732);
 	\draw[thick] (6,-4)--(4,0);
 	\draw[dashed] (6,-4)--(5.4,0.732);
 	\draw[dashed] (6,-4)--(7.4,0.732);
	\draw[thick] (6,-4)--(8,0);
 	\draw[thick] (6,-4)--(4.6,-0.732);
 	\draw[thick] (6,-4)--(6.6,-0.732);	
\end{tikzpicture}
\caption{Minimizers for the Wiener index of simple triangulations and quadrangulations.}
\label{fig:mini}
\end{center}

\end{figure}

There is second research direction of this paper, in addition to the Wiener index.
We  give bounds on the total distance $\sigma(v)$ and the average
distance $\overline{\sigma}(v)$ of a vertex $v$, 
defined as the sum and the average, respectively, of the distances from $v$ to all 
other vertices. Bounds on $\sigma(v)$ were obtained, for example, in \cite{BarEntSze1997}
\cite{EntJacSny1976} and \cite{Zel1968}. Of particular interest is %
the maximum value 
over all $v\in V(G)$ of $\overline{\sigma}(v)$ in a graph $G$,  
usually referred to as %
the {\it remoteness} $\rho(G)$, 
of $G$. 
 It was shown by Zelinka \cite{Zel1968} and, independently, by 
Aouchiche and Hansen \cite{AouHan2011} that the %
remoteness is at most
 $\frac{n}{2}$. %
 For graphs of given
minimum degree $\delta$ these bounds were improved in \cite{Dan2015} by a factor of 
about $\frac{3}{\delta+1}$. For more recent results on %
remoteness see, for example, 
\cite{Dan2016},  and \cite{WuZha2012}.  %

 In this paper we give  sharp upper bounds
on remoteness of  triangulations and quadrangulations  with given connectivity in Corollary~\ref{coro:remote}
and Proposition~\ref{prop:remote}. The bounds are sharp in Proposition \ref{prop:remote} and Corollary \ref{coro:remote} by Figures \ref{fig:t0} through \ref{fig:t5c2} and Figures \ref{fig:t5c3remoteness} through \ref{fig:q3c1}. %
It is not difficult to compute the
distances on those figures from the black vertex to the remaining vertices and show that the sum of distances from the  black vertex meets the upper bound for remoteness. Details will be provided in the Ph.D. dissertation of the third author. Our results show that the maximum remoteness
among  triangulations and quadrangulations of prescribed connectivity $\kappa$ is achieved on those ones that are conjectured to maximize the Wiener index,
{\it except for 5-connected triangulations of order $n=5k+3$}.
There are, however, lots of different realizations of the maximum of remoteness in all classes that we
investigate, except among quadrangulations.

\section{Upper Bounds on the Remoteness of Triangulations and Quadrangulations}

In this section we present bounds on the remoteness of 
triangulations and quadrangulations. A sharp upper bound on the remoteness of 
a triangulation of given order was given by Che and Collins \cite{CheCol2019}.  
We give corresponding bounds for $4$-connected and 
$5$-connected triangulations, as well as for $2$-connected 
and $3$-connected quadrangulations.

We begin by stating a sharp bound on the distance of an arbitrary vertex in a 
$\kappa$-connected graph of given order due to Favaron, Kouider and Mah\'{e}o 
\cite{FavKouMah1989}, from which we will derive some of our bounds.

\begin{prop} {\rm \cite{FavKouMah1989}} \label{prop:bound-per-conn}
Let $G$ be a $\kappa$-connected graph of order $n$, and $x$ an arbitrary vertex
of $G$. Then
\[ \sigma(x) \leq \left\lfloor \frac{n+\kappa-1}{\kappa} \right\rfloor
               \left(n-1-\frac{\kappa}{2}\left\lfloor \frac{n-1}{\kappa}\right\rfloor\right).  \]
\end{prop}

Every simple triangulation is $3$-connected, and every simple
quadrangulation is $2$-connected.  Proposition \ref{prop:bound-per-conn}
yields thus the following  
sharp bounds for the remoteness of $3$-connected and $4$-connected triangulations
and $2$-connected quadrangulations. 

\begin{coro}  \label{coro:remote}
(a) {\rm \cite{CheCol2019}} If $G$ is a simple triangulation 
of order $n$, then
\[ \rho(G) \leq \frac{n+2}{6} + \varepsilon_n,  \] 
where $\varepsilon_n=0$ if $n\equiv 1\pmod{3}$, and 
$\varepsilon_n=\frac{1}{3(n-1)}$ if $n\equiv 0,2 \pmod{3}$. \\
(b) If $G$ is a $4$-connected triangulation  
of order $n$, then
\[ \rho(G) \leq \frac{n+3}{8} + \varepsilon_n,  \]
where $\varepsilon_n=0$ if $n\equiv 1 \pmod{4}$, 
$\varepsilon_n=\frac{3}{8(n-1)}$ if $n\equiv 0,2\pmod{4}$, and 
$\varepsilon_n=\frac{1}{2(n-1)}$ if $n\equiv 3 \pmod{4}$. \\
(c) If $G$ is a simple quadrangulation of order $n$, then
\[ \rho(G) \leq \frac{n+1}{4} + \varepsilon_n,  \] 
where $\varepsilon_n=0$ if $n\equiv 1\pmod{2}$, and 
$\varepsilon_n=\frac{1}{4(n-1)}$ if $n\equiv 0 \pmod{2}$. 
\hfill $\Box$ \\
\end{coro}

Proposition \ref{prop:bound-per-conn} also yields good bounds
for the remoteness of $5$-connected triangulations and $3$-connected 
quadrangulations. These bounds are however not sharp for all values of $n$.
In order to obtain sharp bounds we need some additional terminology 
and results from \cite{AliDanMuk2012}. 

Let $v$ be a fixed vertex of a connected plane graph and $i\in \mathbb{N}$ with $i<e(v)$. 
We say that a vertex $w\in N_i(v)$ is \emph{active} if it has a neighbor
in $N_{i+1}(v)$.

\begin{la} {\rm \cite{AliDanMuk2012}}  \label{la:active-vertices}
Let $G$ be a $3$-connected plane graph, $v$ a vertex of $G$ and $i\in \mathbb{N}$
with $1 \leq i \leq e(v)-1$. For every active vertex $w \in N_i(v)$ there exist
two other active vertices $w', w'' \in N_i(v)$ such that $w$ and $w'$ share a 
face of $G$, and $w$ and $w''$ also share a face of $G$. 
\end{la}

\begin{la}   \label{la:tail-of-distance-sequence} 
(a) Let $G$ be a $5$-connected simple triangulation, $v$ a vertex of $G$ and $d=e_G(v)$.
If $n_{d-1}(v)=5$, then $n_d(v)=1$. \\
(b) Let $G$ be a $3$-connected simple quadrangulation, $v$ a vertex of $G$ and $d=e_G(v)$. 
If $n_{d-1}(v)=3$, then $n_d(v)=1$. If $n_{d-2}(v)=3$ and $n_{d-1}(v)=4$, then $n_d(v)> 1$. 
\end{la}

\begin{proof}
(a) Assume that $G$ is a $5$-connected simple triangulation, $v$ is a vertex of 
$G$, and $n_{d-1}=5$, where $d$ is the eccentricity of $v$. 
This implies that $N_{d-1}$ is a minimum cutset of $G$. Hence, since $G$ is a triangulation,
$N_{d-1}$ induces a cycle $C$ of length $5$. 
We first show that 
\begin{equation} \label{eq:5-conn-tri-in-out-C}
\textrm{the vertices in $N_d$ are all inside $C$, or all outside $C$.}
\end{equation} 
Suppose not. Then there exist vertices $a,b \in N_d$ such that $a$ is inside $C$,
and $b$ is outside $C$. Since $G$ is $5$-connected, there exist a
$(v,N_{d-1})$-fan $F_v$, an $(a,N_{d-1})$-fan $F_a$, and a $(b,N_{d-1})$-fan $F_b$. 
Any two of these three fans share only the vertices of $N_{d-1}$. Indeed, 
other than vertices in $N_{d-1}$, fan $F_v$ contains only
vertices in $\bigcup_{i=0}^{d-1} N_i$, while fan $F_a$ contains only 
vertices in $N_{d-1} \cup N_d$ that are inside $C$, while 
fan $F_b$ contains only vertices in $N_{d-1} \cup N_d$ that are outside $C$.   
Now contracting the edges in $F_a-N_{d-1}$, the edges in $F_b-N_{d-1}$, 
and the edges in $F_{v}-N_{d-1}$ to three single vertices yields a graph that
contains $3K_1 + C_5$ as a subgraph. Hence $G$ contains $3K_1+C_5$ as a minor. 
Contracting three consecutive vertices of the $5$-cycle shows that this implies
that $G$ contains $K_3 +3K_1$ as a minor, which contradicts the planarity of $G$.
This contradiction proves \eqref{eq:5-conn-tri-in-out-C}.

By \eqref{eq:5-conn-tri-in-out-C} we may assume that all 
vertices of $N_d$ are inside the cycle $C$. Since every vertex of $N_{d}$ is 
adjacent to some vertex of $N_{d-1}$, the subgraph $G[N_d]$ is outerplanar.
Hence 
\begin{equation}  \label{eq:max-size-of-G[Nd]}
 m(G[N_d]) \leq \left\{ \begin{array}{cc}
         0 & \textrm{if $n_d=1$,} \\
         1 & \textrm{if $n_d =2$,} \\ 
         2n_d-3 & \textrm{if $n_d\geq 3$.} 
            \end{array} \right.
\end{equation}
We now bound the sum of the degrees of the vertices in $N_d$. 
Let $H$ be the plane graph obtained from $G[N_{d-1} \cup N_d]$ by adding 
a new vertex $z$ in the outer face of $C$ and joining it to all five vertices
of $C$. Then $H$ has order $n(H)=1+n_{d-1} + n_d = n_d+6$. Since $H$ is a plane 
graph we have 
$m(H) \leq 3n(H)-6 \leq 3n_d+12$. At least $10$ edges of $H$ are incident with
$z$ or belong to $C$, and are thus not incident with any vertex of $N_d$, so they
don't contribute to the sum of the degrees of vertices in $N_d$. 
Since the edges of $G[N_d]$ contribute two to the sum of the degrees of vertices
in $N_d$, we have 
\[ \sum_{x \in N_d} {\rm deg}_G(x) 
      \leq  (m(H)-10) + m(G[N_d]) 
       \leq \left\{ \begin{array}{cc}
         (3n_d+2)+0 & \textrm{if $n_d=1$,} \\
         (3n_d+2) + 1 & \textrm{if $n_d=2$,} \\         
         (3n_d+2)+(2n_d-3) & \textrm{if $n_d\geq 3$.}          
          \end{array} \right. \]
It is easy to verify that this implies $\sum_{x\in N_d} {\rm deg}_G(x) < 5n_d$ 
whenever $n_d> 1$. But since $G$ is $5$-connected, every vertex of $G$ has
degree at least five. Hence we conclude that $n_d=1$, which proves (a). 

(b) Let $G$ be a $3$-connected simple quadrangulation, $v$ a vertex of $G$, 
and $d=e(v)$. \\
To prove the first statement assume that $n_{d-1}=3$. Let $N_{d-1}(v)=\{w, w', w''\}$. 
Since $G$ is a quadrangulation and thus bipartite, the set $\{w,w,w''\}$ is 
independent in $G$. 
Since $G$ is $3$-connected, the vertices $w, w', w''$ have a neighbor in $N_{d}$
and are thus active. 
By Lemma \ref{la:active-vertices}, $w$ and $w'$ share a face, 
and so do $w$ and $w''$, as well as $w'$ and $w''$. 
Hence we can add edges $ww'$, $ww''$ and 
$ww''$ to $G$ to obtain a plane graph (but not a quadrangulation). 
Let $C$ be the cycle consisting of the edges $ww',w'w'',w''w$. A proof similar
to that in (a) shows that the vertices of $N_d$ are all inside $C$, or all outside
$C$. Without loss of generality we assume the former.  
We now bound the sum of the degrees of the vertices in $N_d$. \\
Let $H$ be the plane graph obtained from $G[N_{d-1} \cup N_d]+E(C)$ by adding 
a new vertex $z$ in the outer face of $C$ and joining it to all three vertices
of $C$. 
Since $G$ is a quadrangulation, the only faces of $H$ of length three
are the six faces that have one of the three edges of $C$ on their
boundary. Let $H'$ be the plane graph $H-E(C)=G[N_{d-1}\cup N_d]$. Then
$n(H')=n_{d-1}+n_d+1=n_d+4$ and, since $H'$ has only faces of length at
least four,  $m(H') \leq 2n(H')-4 = 2n_d+4$. \\  
Exactly three edges of $H$ are incident with
$z$ and are thus not incident with any vertex of $N_d$. 
Since $G$ is bipartite, $G[N_d]$ contains no edges. Hence  
\[ \sum_{x \in N_d} {\rm deg}_G(x) 
      =  (m(H')-3)  
       \leq 2n_d+1. \]
This implies $\sum_{x\in N_d} {\rm deg}_G(x) < 3n_d$ 
whenever $n_d> 1$. But since $G$ is $3$-connected, every vertex of $G$ has
degree at least three. Hence we conclude that $n_d=1$, which proves the first 
statement of (b). \\
To prove the second statement of (b) assume that $n_{d-2}=3$ and $n_{d-1}=4$. 
Suppose to the contrary that $n_d=1$. Let $N_{d-2}=\{w,w',w''\}$. 
The same arguments as in the proof of the first statement of (b) show that 
we can add the edges $ww', ww'', w'w''$ to $G$ to obtain a plane graph, 
such that these three edges form a cycle $C$, and that the vertices in $N_{d-1}\cup N_d$
are either all inside $C$ or all outside $C$, without loss of generality we assume the former. 
Let $H$ be the plane graph obtained from $G[N_{d-2} \cup N_{d-1} \cup N_d]+E(C)$ by adding 
a new vertex $z$ in the outer face of $C$ and joining it to all three vertices
of $C$. 
Since $G$ is a quadrangulation, the only faces of $H$ of length three
are the six faces that have one of the three edges of $C$ on their
boundary. Let $H'$ be the plane graph $H-E(C)=G[N_{d-2} \cup N_{d-1}\cup N_d]$. Then
$n(H')=n_{d-2}+n_{d-1}+n_d+1=9$ and, since $H'$ has only faces of length at
least four,  $m(H') \leq 2n(H')-4 = 14$. Exactly three edges of $H'$ are incident
with $z$ and thus not incident with vertices in $N_{d-1}$. Since $G$ is bipartite, 
no edge joins two vertices of $N_{d-1}$, and so we have
$\sum_{x\in N_{d-1}} {\rm deg}_G(x) \leq 11 < 3n_{d-1}$.
Therefore, $N_{d-1}$ contains a vertex of degree less than three in $G$, which 
contradicts $G$ being $3$-connected.  The second statement of (b) follows. 
\end{proof}

For the remaining proof of this section we define the function $F$ which assigns to a
finite sequence $X=(x_0,x_1,\ldots,x_k)$ of integers the value 
$F(X)=\sum_{i=0}^k ix_i$. So if $v$ is a vertex of eccentricity $d$ in a connected graph $G$, then
$\sigma(v) = \sum_{i=0}^d in_i(v) = F(n_0,n_1,\ldots,n_d)$.

\begin{prop} \label{prop:remote}
(a) Let $G$ be a $5$-connected triangulation of order $n$. Then
\[ \rho(G) \leq \frac{n+4}{10} + \varepsilon_n,  \]
where 
$\varepsilon_n= - \frac{3}{5(n-1)}$ if $n\equiv 0 \pmod{5}$, 
$\varepsilon_n= - \frac{1}{n-1}$ if $n\equiv 1\pmod{5}$, 
$\varepsilon_n= \frac{2}{5(n-1)}$ if $n\equiv 2 \pmod{5}$,
and 
$\varepsilon_n=-\frac{2}{5(n-1)}$ if $n\equiv 3,4 \pmod{5}$. \\
(b) If $G$ is a $3$-connected quadrangulation of order $n$, then 
\[ \rho(G) \leq \frac{n+2}{6} + \varepsilon_n,  \]
where 
$\varepsilon_n= - \frac{5}{3(n-1)}$ if $n\equiv 0 \pmod{3}$, 
$\varepsilon_n=-\frac{1}{n-1}$ if $n\equiv 1\pmod{3}$, and 
$\varepsilon_n= \frac{1}{3(n-1)}$ if $n\equiv 2\pmod{3}$. 
\end{prop}

\begin{proof}

(a) It suffices to show that for an arbitrary vertex $v$ of $G$ we have
\[ \sigma(v) \leq \frac{n^2+3n}{10} + \varepsilon_n',  \]
where
$\varepsilon_n' = -10$ if $n\equiv 0 \pmod{5}$,
$\varepsilon_n' = -14$ if $n\equiv 1 \pmod{5}$,
$\varepsilon_n' = 0$ if $n\equiv 2 \pmod{5}$, and 
$\varepsilon_n' = -8$ if $n\equiv 3,4 \pmod{5}$. \\
Fix $v \in V(G)$ and let $d=e(v)$. Then 
\[ \sigma(v) = \sum_{i=0}^d in_i = F(n_0,n_1,\ldots,n_d).  \]
All $n_i$ are positive integers, $n_0=1$ and $\sum_{i=0}^d n_i=n$. 
Since $G$ is $5$-connected we also have 
$n_i\geq 5$ for all $i\in \{1,2,\ldots,d-1\}$. 
To bound $F(n_0, n_1,\ldots,n_d)$ from above we assume that $n$ is fixed, and that
$d' \in \mathbb{N}$ and $X_{max}(n)=(n_0',n_1',\ldots, n_{d'}')$ maximise the function
$F$ among all integers $d$ and sequences $X$ that satisfy these constraints. 
We first note that $n_1'=n_2' =\cdots=n_{d-1}'=5$. Indeed, if  
$n_i'> 5$ for some $i$ with $1 \leq i \leq d'-1$, 
then decreasing $n_i'$ by $1$ and increasing $n_{i+1}'$ by $1$ yields a 
new sequence $X'$ that satisfies the above constraints and for which 
$F(X')=F(X_{max}(n))+1$, contradicting the choice of $X_{max}(n)$. Also, if $n_{d'}'> 5$, then
decreasing $n_{d'}'$ by $1$, appending a new entry $n_{d'+1}'=1$ at
the end and increasing $d'$ by $1$ yields a sequence that satisfies the 
requirement but whose $F$-value is greater, again a contradiction to the
choice of $X_{max}(n)$. Therefore, if $q$ and $r$ are positive integers with $1 \leq r \leq 5$
such that $n-1=5q+r$, then the unique sequence maximising $F$ subject to the
above constraints is 
\[ X_{max}(n) = (1, 5, 5, \ldots,5,r), \] 
where the entry $5$ appears exactly $q$ times. If $r\neq 1$, then it is easy to see 
that the unique sequence with the second largest $F$-value satisfying the constraints is
the sequence 
\[ X_{max}'(n) = (1, 5, 5, \ldots,5, 6,r-1), \]  
where the entry $5$ appears exactly $q-1$ times. \\[1mm]
\textsc{ Case 1:} $n\equiv 2 \pmod{5}$. \\
Then $F(n_0,n_1,\ldots,n_d) \leq F(X_{max}(n)) = \frac{1}{10}(n^2+3n)$, as desired. \\[1mm]
\textsc{ Case 2:} $n \equiv 0,1,3,4 \pmod{5}$. \\ 
Then $(n_0,n_1,\ldots,n_d) \neq X_{max}(n)$ since otherwise, if 
$(n_0,n_1,\ldots,n_d) = X_{max}(n)$, then $n_{d-1}=5$ and $n_d=r\neq 1$, contradicting 
Lemma \ref{la:tail-of-distance-sequence}(a). Therefore, 
$F(n_0,n_1,\ldots,n_d) \leq F(X_{max}'(n))$, and a simple calculation shows
that $F(X_{max}')(n)$ is the claimed upper bound on $\sigma(v)$. \\
(b) The proof of (b) is analogous to that of (a), with only two differences: 
The condition $n_i\geq 5$ for all $i\in \{1,2,\ldots,d-1\}$ in (a) is 
replaced by $n_i\geq 3$ for all $i\in \{1,2,\ldots,d-1\}$. 
Also, Lemma \ref{la:tail-of-distance-sequence}(b) implies that 
for $n\equiv 1,2 \pmod{3}$ we have $(n_0,n_1,\ldots,n_d) \neq X_{max}(n)$ and so
$F(n_0,n_1,\ldots,n_d) \leq F(X_{max}(n)')$,  
while for $n\equiv 0 \pmod{3}$ Lemma \ref{la:tail-of-distance-sequence}(b)
implies that  $(n_0,n_1,\ldots,n_d) \neq X_{max}(n), X_{max}(n)'$ and thus
$F(n_0,n_1,\ldots,n_d) < F(X_{max}(n)')$. 
\end{proof}

\section{Upper Bounds on the Wiener Index of Triangulations and Quadrangulations}

In this section we present asymptotically sharp upper bounds on the Wiener index of simple
triangulations and simple quadrangulations, and improved bounds for simple $4$-connected
and $5$-connected triangulations as well as simple $3$-connected quadrangulations. 

In the statements and proofs of our results we use the following notation. 
If $S$ is a separating cycle of a plane graph $G$, then we denote the set
of vertices inside $S$ by $A$, and the set of vertices outside $S$ by $B$. 
We often use $S$ also for the set of vertices on this cycle, and we
further let $a:=|A|$, $b:=|B|$ and $s:=|S|$. The following separator theorem
by Miller is an important tool for the proof of our bounds.

\begin{theo} {\rm (\cite{Mil986})} \label{theo:separating-cycle}
If $G$ is a $2$-connected plane graph of order $n$ whose faces have length at most $\ell$, 
then $G$ has a separating cycle $S$ of length at most  
$2\sqrt{2\lfloor \ell/2\rfloor n}$, such that 
$a,b \leq \frac{2}{3}n$. 
\end{theo}

We now define a plane graph which will be used in the proof of the
main result of this section.

\begin{defi}  \label{defi:Fp}
For $p\in \mathbb{N}$ with $p\geq 3$ let $F_p$ be the 
plane graph constructed as follows. 
Let $C=u_0, u_1, \ldots,u_{p-1}, u_0$ be a cycle of length $p$.  
Inside $C$ we add  a cycle  $C'=v_0, v_1,\ldots, v_{2p-1}v_0$ of length 
$2p$ and edges $u_iv_{2i-1}, u_iv_{2i}$, $u_iv_{2i+1}$ for $i=0,1,\ldots,p-1$, 
with indices taken modulo $p$ for the $u_i$ and modulo $2p$ for the $v_i$.
Inside $C'$ we add a cycle  $C''= w_0, w_1,\ldots, w_{2p-1}, w_0$ of length $2p$ 
and edges 
$v_iw_i, v_iw_{i+1}$ for $i=0,1,\ldots,2p-1$, with all indices taken modulo $2p$.
Inside $C''$ we add a new vertex $z$ and join it
to every vertex of $C''$. The graph $F_4$ is shown in Figure \ref{fig:graph-H4}. \\
We define $F_p'$ to be a plane graph with the same vertex and edge set as
$F_p$, but with the cycle $C'$ outside the cycle $C$, the cycle $C''$ outside
the cycle $C'$, and $z$ lying in the unbounded face whose boundary is $C''$. 
\end{defi}

  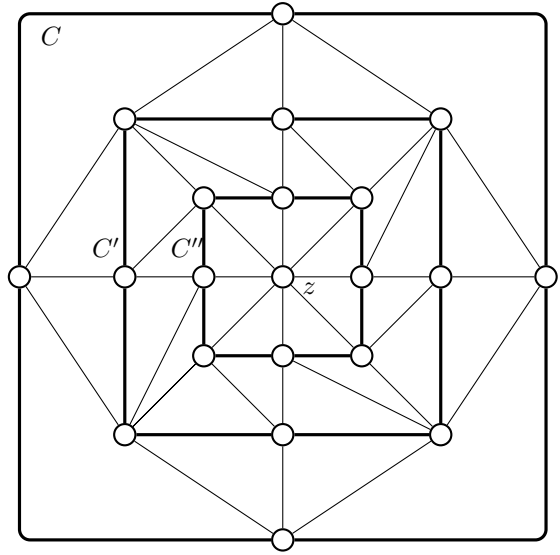
\begin{figure}[htbp]
  \begin{center}
\begin{tikzpicture}
  [scale=0.7,inner sep=1mm, %
   vertex/.style={circle,thick,draw}, %
   thickedge/.style={line width=2pt}] %
    \begin{scope}[>=triangle 45]
    
     \node[vertex] (a1) at (5,0) [fill=white] {};
     \node[vertex] (b1) at (2,2) [fill=white] {};
     \node[vertex] (b2) at (5,2) [fill=white] {};
     \node[vertex] (b3) at (8,2) [fill=white] {};
     \node[vertex] (c1) at (3.5,3.5) [fill=white] {};
     \node[vertex] (c2) at (5,3.5) [fill=white] {};
     \node[vertex] (c3) at (6.5,3.5) [fill=white] {};
     \node[vertex] (d1) at (0,5) [fill=white] {};
     \node[vertex] (d2) at (2,5) [fill=white] {};
     \node[vertex] (d3) at (3.5,5) [fill=white] {};          
     \node[vertex] (d4) at (5,5) [fill=white] {};
     \node[vertex] (d5) at (6.5,5) [fill=white] {};
     \node[vertex] (d6) at (8,5) [fill=white] {};
     \node[vertex] (d7) at (10,5) [fill=white] {};
     \node[vertex] (e1) at (3.5,6.5) [fill=white] {};
     \node[vertex] (e2) at (5,6.5) [fill=white] {};
     \node[vertex] (e3) at (6.5,6.5) [fill=white] {};
     \node[vertex] (f1) at (2,8) [fill=white] {};
     \node[vertex] (f2) at (5,8) [fill=white] {};
     \node[vertex] (f3) at (8,8) [fill=white] {};
     \node[vertex] (g1) at (5,10) [fill=white] {};
    \draw[very thick] (b1)--(b2)--(b3)--(d6)--(f3)--(f2)--(f1)--(d2)--(b1);  
    \draw[very thick] (c1)--(c2)--(c3)--(d5)--(e3)--(e2)--(e1)--(d3)--(c1);     
    \draw[black] (a1)--(b3)--(d7)--(f3)--(g1)--(f1)--(d1)--(b1)--(a1)
          (a1)--(b2)   (d7)--(d6)  (g1)--(f2)  (d1)--(d2);
    \draw[black]  (b1)--(c1)--(b2)--(c2)--(b3)--(c3)--(d6)--(d5)--(f3)
    (f3)--(e3)--(f2)--(e2)--(f1)--(e1)--(d2)--(d3)--(b1)--(c1);
    \draw[rounded corners, very thick] (a1)--(10,0)--(d7) 
       (d7)--(10,10)--(g1) (g1)--(0,10)--(d1)  (d1)--(0,0)--(a1);
    \draw[black] (c1)--(d4)--(c2)  (c3)--(d4)--(d5)  (e3)--(d4)--(e2)
        (e1)--(d4)--(d3);          
    \node [right, below] at (5.5,5.05) {$z$};    
    
    \node [below] at (0.6,9.9) {$C$}; 
    \node [below] at (1.65,5.9) {$C'$};    
    \node [below] at (3.2,5.9) {$C''$};

    \end{scope}
\end{tikzpicture}   
\caption{The graph $F_p$ for $p=4$. The $p$-cycle $C$ and $2p$-cycles $C'$ and $C''$ drawn with thick lines.}
\label{fig:graph-H4}
\end{center}
\end{figure}

\begin{la}   \label{la:Hp-is-5-c}
Let $F_p$ be the graph defined in Definition \ref{defi:Fp} above. \\
(a) $\kappa(F_p) \geq 5$ for $p\geq 3$. \\
(b) If $u \in V(F_p)$ and $M \subseteq V(C)$ with $|M| \leq 5$,
then $F_p$ contains a $(u,M)$-fan. \\
(c) If $M_1, M_2 \subseteq V(C)$ are two sets with $|M_1| = |M_2| \leq 5$, 
then $F_p$ contains a set of $|M_1|$ disjoint paths from $M_1$ to $M_2$. 
\end{la}

\begin{proof}
(a) It is easy to verify that any two vertices of $F_p$ are joined by 
five internally disjoint paths, hence $F_p$ is $5$-connected.  \\
(b) and (c) follow directly from $F_p$ being $5$-connected. 
\end{proof}

\begin{theo}   \label{theo:kpa-conn-tri}
Let $\kappa\in \{3,4,5\}$. 
Then there exists a constant $D$ such that 
\[ W(G) \leq \frac{1}{6\kappa}n^3 + D n^{5/2} \] 
for every $\kappa$-connected simple triangulation of order $n$.
\end{theo}

\begin{proof}
Our proof is by induction on $n$. Define $D:=\max\{D_1,D_2\}$, where
$D_1$ is the smallest real $x$ for which the inequality 
$ W(G) \leq \frac{1}{6\kappa}n^3 + x n^{5/2}$ holds for all 
$\kappa$-connected simple 
triangulations $G$ of order at most $10^4$, and $D_2$ is the smallest real
$x$ for which 
$8.1 + 0.76 x \leq x$ holds. 
We prove by induction on $n$ that for all simple triangulations $G$
of order $n$, 
\begin{equation} \label{eq:induction-st} 
W(G) \leq \frac{1}{6\kappa}n^3 + D n^{5/2}, 
\end{equation} 
Now \eqref{eq:induction-st} holds for all $n\leq 10^4$ by the 
choice of $D$. Let $n>10^4$. By our induction hypothesis we may assume 
that \eqref{eq:induction-st} holds for all graphs of order
less than $n$. 
 
Since $G$ is $2$-connected, it follows by Theorem \ref{theo:separating-cycle} 
that $G$ contains a separating cycle $S=t_0t_1\ldots t_{s-1}t_0$ 
with $a,b \leq \frac{2}{3}n$,  
where $A, B, a, b, s$ are as in Theorem \ref{theo:separating-cycle} and above it.
Let $H$ be the simple triangulation obtained from the plane graph $G-A$ 
as follows. We first delete all edges between non-consecutive vertices of $S$ that 
run inside the cycle $S$. 
Inside $S$ we insert the graph $F_s$ by identifying the cycles $S$ and $C$, specifically  $t_i \in S$  %
with $u_i \in V(F_s)$ for $i=0,1,\ldots,s-1$. Clearly, $H$ is a simple 
triangulation of order $b+5s+1$. 
Similarly let $K$ be the simple 
triangulation of order $a+5s+1$ obtained from the 
plane graph $G-B$ by deleting all edges between non-consecutive vertices of
$S$ that run outside the cycle $S$ and inserting $F_s'$  (a copy of $F_s$) into the 
unbounded face, bounded by the vertices of $S$, by identifying $t_i \in S$ with
$u_i \in V(F_s')$ for $i=0,1,\ldots,s-1$.

For an illustration, see Figure~\ref{fig:graph-H4}.  
We claim that 
\begin{equation}  \label{eq:H-K-kappa-conn} 
\textrm{$H$ and $K$ are $\kappa$-connected.}
\end{equation}
We prove \eqref{eq:H-K-kappa-conn} only for $H$, the proof for
$K$ is analogous. Let $u,v$ be two arbitrary vertices of $H$. It suffices
to show that there
exist $\kappa$ internally disjoint $(u,v)$-paths in $H$. First assume that 
both, $u$ and $v$, are in  $V(F_s)$, then it follows from 
Lemma \ref{la:Hp-is-5-c}(a) and $\kappa \leq 5$ that there are
$\kappa$ internally disjoint $(u,v)$-paths in $F_s$, and thus in $H$.  
Now assume that exactly one of the two vertices, say $u$, is in $V(F_s)$.
Fix a vertex $w \in A$. 
It follows from the $\kappa$-connectedness of $G$ that in $G$ there exist
$\kappa$ internally disjoint $(w,v)$-paths $P_1,P_2,\ldots,P_{\kappa}$. 
For $i=1,2,\ldots,\kappa$ let $w_i$ be the last vertex of $P_i$ on $C$, 
and let $P_i'$ be the $(w_i,v)$-section of $P_i$. By 
Lemma \ref{la:Hp-is-5-c}(b), $F_s$ contains
a $(u,\{w_1,\ldots,w_{\kappa}\})$-fan $F$. Then $F$ together with $P_1',\ldots,P_{\kappa}'$ yields a collection of $\kappa$ internally disjoint
$(u,v)$-paths in $H$.  
Finally assume that both, $u$ and $v$, are not in $V(F_s)$. 
Then it follows from the $\kappa$-connectedness of $G$ that there exists
internally disjoint $(u,v)$-paths $P_1,P_2,\ldots,P_{\kappa}$ in $G$. 
If none of these contains a vertex in $V(F_s)$, then $P_1,P_2,\ldots,P_{\kappa}$ 
form a collection of $\kappa$ internally disjoint $(u,v)$-paths in $H$. 
If some of the paths, 
$P_1,\ldots,P_k$ say, contain a vertex of $V(F_s)$, then let 
$a_i$ and $a_i'$ be the first and last vertex, respectively, of 
$P_i$ in  $V(F_s)$. Let $M=\{a_1,\ldots,a_k\}$ and $M'=\{a_1',\ldots,a_k'\}$.
By Lemma \ref{la:Hp-is-5-c}(c), $F_s$ contains $k$ disjoint paths
$Q_1,\ldots,Q_k$ from $M$ to $M'$. Then the $(u,a_i)$-sections and the $(a_i',v)$-sections
of the paths $P_i$ together with $Q_1,\ldots,Q_k$ and the paths
$P_{k+1},\ldots,P_{\kappa}$ form a collection of $\kappa$ 
internally disjoint $(u,v)$-paths in $H$. This proves 
\eqref{eq:H-K-kappa-conn}.

The two graphs $H$ and $K$ have exactly the vertices in $S$ in
common. We now bound the Wiener index of $G$ in terms of the Wiener indices 
of $H$ and $K$, and the total distance of $z$ in $H$ and $K$.   
\begin{eqnarray}
W(G) & <  & \sum_{ \{x,y\} \subseteq B \cup S} d_G(x,y) 
            +  \sum_{ \{x,y\} \subseteq A \cup S} d_G(x,y) 
            +  \sum_{ x\in A, \ y\in B} d_G(x,y)  \nonumber \\
     & < & {n \choose 2} \frac{s}{2} 
         + \hspace*{-0.8em} \sum_{ \{x,y\} \subseteq B \cup V(F_s)} \hspace*{-0.8em}  d_H(x,y) 
             + \hspace*{-0.8em}  \sum_{ \{x,y\} \subseteq A \cup V(F_s)} \hspace*{-0.8em}  d_K(x,y) 
            +  \hspace*{-0.8em}  \sum_{ x\in A, \ y\in B}  \hspace*{-0.8em} \big( d_H(x,z) + d_K(z,y) \big).   \label{eq:W-total}
\end{eqnarray}
Indeed, for any two vertices $x$ and $y$ of $G$ that are both in $B \cup S$, 
we have $d_G(x,y) \leq d_H(x,y)+\frac{s}{2}$ since a shortest $(x,y)$-path in
$H$ either contains only vertices in $B\cup S$, in which case it is also
a path in $G$, or it contains vertices in $V(F_s)-S$, in which case 
replacing the segment between the first and last occurrence of a vertex in $V(F_s)-S$ in the 
path by a segment of the
cycle $S$ that contains at most $s/2$ vertices yields an $(x,y)$-path in $G$. 
Similarly, if $x$ and $y$ are both in $A \cup S$, then $d_G(x,y) \leq d_K(x,y)+\frac{s}{2}$. 
Finally, if $x \in A$ and $y\in B$, then we can obtain an $(x,y)$-path in $G$ from
the concatenation of an $(x,z)$-path in $K$ and a $(z,y)$-path in $H$ by
replacing $z$ with a segment of $S$ containing at most $s/2$ vertices. This 
proves \eqref{eq:W-total}. 

We now bound each of the terms in \eqref{eq:W-total}. Since $H$ and $K$ are 
$\kappa$-connected simple triangulations  
of order $b+5s+1$ and $a+5s+1$, respectively, we have by induction 
\begin{equation}    \label{eq:W(H)} 
\sum_{ \{x,y\} \subseteq B \cup V(F_s)} d_H(x,y)  =   W(H) 
    \leq \frac{1}{6\kappa} (b+5s+1)^3 +D(b+5s+1)^{5/2}, 
\end{equation}
and
\begin{equation}    \label{eq:W(K)} 
\sum_{ \{x,y\} \subseteq A \cup V(F_s)} d_K(x,y)  =   W(K) 
    \leq \frac{1}{6\kappa} (a+5s+1)^3 +D(a+5s+1)^{5/2}.
\end{equation}
It follows from the bounds on remoteness in Corollary \ref{coro:remote}(a),(b) 
and Proposition \ref{prop:remote} 
that $\sigma(v) \leq \frac{1}{2\kappa}n^2 + \frac{\kappa-2}{2\kappa}n
   + \frac{\kappa-3}{2\kappa}$ 
for every vertex $v$ of a $\kappa$-connected triangulation of order $n$. Hence 
$\sigma(z,H) \leq \frac{1}{2\kappa}(a+5s+1)^2 + \frac{\kappa-2}{2\kappa}(a+5s+1)
      + \frac{\kappa-3}{2\kappa}$ and 
$\sigma(z,K) \leq \frac{1}{2\kappa}(b+5s+1)^2 + \frac{\kappa-2}{2\kappa}(b+5s+1)
      + \frac{\kappa-3}{2\kappa}$. Hence 
\begin{eqnarray*}
\sum_{ x\in A, \ y\in B}  \big( d_H(x,z) + d_K(z,y) \big)
& = & b\, \sum_{x\in A} d_H(x,z) + a\, \sum_{y\in B} d_K(z,y)  \\
 & \hspace*{-20em}< & \hspace*{-10em} b\, \sigma(z,H) + a\, \sigma(z,K)  \\
& \hspace*{-20em}\leq & \hspace*{-10em}
      \frac{b}{2\kappa} \big[ (a+5s+1)^2 + (\kappa-2)(a+5s+1) + \kappa-3 \big]  \\
& \hspace*{-20em} & \hspace*{-10em}      
        + \frac{a}{2\kappa}\big[(b+5s+1)^2 + (\kappa-2)(b+5s+1) + \kappa-3 \big], 
\end{eqnarray*}
and since $a<a+5s+1$ and $b<b+5s+1$, 
\begin{eqnarray}
\sum_{ x\in A, \ y\in B}   d_H(x,z) + d_K(z,y) 
  &<& \frac{1}{2\kappa}(a+5s+1)^2(b+5s+1) + \frac{1}{2\kappa}(a+5s+1)(b+5s+1)^2 \nonumber \\
      & &   + \frac{\kappa-2}{\kappa}(a+5s+1)(b+5s+1)
             + \frac{\kappa-3}{2\kappa}(a+b). \label{eq:distances-A-and-B} 
\end{eqnarray}  
Hence we obtain from \eqref{eq:W-total}, \eqref{eq:W(H)}, \eqref{eq:W(K)}.
and \eqref{eq:distances-A-and-B}, 
\begin{eqnarray}
W(G) & < & \frac{1}{6\kappa} (a+5s+1)^3 +D(a+5s+1)^{5/2} 
     + \frac{1}{6\kappa} (b+5s+1)^3   + D(b+5s+1)^{5/2}  \nonumber \\
   & &   +
     \frac{1}{2\kappa}(a+5s+1)^2(b+5s+1) + \frac{1}{2\kappa}(a+5s+1)(b+5s+1)^2
      + {n \choose 2} \frac{s}{2}  \nonumber \\
   & &    + \frac{\kappa-2}{\kappa}(a+5s+1)(b+5s+1)
             + \frac{\kappa-3}{2\kappa}(a+b) \nonumber \\
   & = & \frac{1}{6\kappa} (a+b+10s+2)^3 + D \big[ (a+5s+1)^{5/2} + (b+5s+1)^{5/2}\big] 
        \nonumber \\
   & &    + \frac{\kappa-2}{\kappa}(a+5s+1)(b+5s+1)
             + \frac{\kappa-3}{2\kappa}(a+b) + {n \choose 2} \frac{s}{2}.  \label{eq:W-total-2} 
\end{eqnarray}
We bound the terms of the right hand side of \eqref{eq:W-total-2} separately.
We make use of the facts that $a+b+s=n$, and that by 
Theorem \ref{theo:separating-cycle} in conjunction with $n>10^4$ we have 
$s \leq 2^{3/2}n^{1/2} < 0.03n-1$. 
We bound the first term of  \eqref{eq:W-total-2} by 
$(a+b+10s+2)^3 = (n+9s+2)^3 \leq (n+ 9\cdot 2^{3/2}n^{1/2} + 2)^2$.
To bound the second term note that the real function $f(x)=x^{5/2}$ is 
concave up and that $a,b \leq \frac{2}{3}n$ by 
Theorem \ref{theo:separating-cycle}, which implies that 
$(a+5s+1)^{5/2} + (b+5s+1)^{5/2}$ is maximised if $a=\frac{2}{3}n$ 
and $b= \frac{1}{3}n -s$ (or vice versa). Therefore,  
$(a+5s+1)^{5/2} + (b+5s+1)^{5/2} 
    \leq (\frac{2}{3}n+5s+1)^{5/2} + (\frac{1}{3}n+4s+1)^{5/2}
   \leq  (\frac{2}{3}n+0.15n)^{5/2} + (\frac{1}{3}n+0.12n)^{5/2}
     = \big( (\frac{2}{3}+0.15)^{5/2} + (\frac{1}{3} + 0.12)^{5/2} \big)n^{5/2}
     <0.76n^{5/2}$. 
To bound the third term note that $\frac{\kappa-2}{\kappa}<1$, 
$a+5s+1<n$ and $b+5s+1<n$, so $\frac{\kappa-2}{\kappa}(a+5s+1)(b+5s+1)<n^2$.   
To bound the fourth term note that $\frac{\kappa-3}{2\kappa}<1$ and
$a+b<n$, so $\frac{\kappa-3}{2\kappa}(a+b)<n$. 
Finally,  ${n \choose 2} < \frac{1}{2}n^2$, and so we bound the fifth term
by  ${n \choose 2} \frac{s}{2} < 2^{-1/2}n^{5/2}$.  
In total we obtain from 
\eqref{eq:W-total-2},
\begin{eqnarray*}
W(G)  &<&  \frac{1}{6\kappa} (n+ 9\cdot 2^{3/2}n^{1/2} + 2)^3 
           + 0.76\; D n^{5/2}  + n^2 + n +2^{-1/2}n^{5/2} \\
    &=& \frac{1}{6\kappa}n^3 
     + \big( \frac{13}{\kappa} + 0.76\, D + 1 + 2^{-1/2} \big) n^{5/2} 
     + \frac{338}{\kappa} n^2 + \frac{8788}{3\kappa}n^{3/2}.
\end{eqnarray*}

Since $n\geq 10^4$, we have 
$\frac{338}{\kappa}n^2  + \frac{8788}{3\kappa}n^{3/2}  < 2n^{5/2}$.
Also, $ \frac{13}{\kappa} + 0.76\, D + 1 + 2^{-1/2} < 6.1 + 0.76\, D$, 
and so
\begin{eqnarray*}
W(G) &<& \frac{1}{6\kappa}n^3  + \big(8.1 + 0.76\, D\big)n^{5/2} \\
    & \leq & \frac{1}{6\kappa}n^3  + D n^{5/2} 
\end{eqnarray*}
since $D$ satisfies $8.1 + 0.76\, D \leq D$. The theorem follows. 
\end{proof}

The following bound on the Wiener index of simple 
quadrangulations is proved in a similar way. The only difference is that
a slightly modified version $Q_p$ of the plane graph $F_p$ 
is used in the proof. For an even $p$ with $p\geq 4$ let $Q_p$ be
the plane graph obtained from a cycle 
$C=u_0, u_1,\ldots,u_{p-1},u_0$ of length $p$, inside which
we add a cycle $C'=v_0, v_1,\ldots,v_{p-1},v_0$ of length $p$ 
and edges $u_iv_i$ for $i=0,1,\ldots,p-1$, inside which we 
add a vertex $z$ and joint it to all $v_i$ with $i$ even. 
It is easy to verify that  a 3-connected quadrangualation with the insertion of $Q_p$ stays $3$-connected. 
Apart from this difference, the proof of Theorem
\ref{theo:kappa-conn-quad} follows closely that 
of Theorem \ref{theo:kpa-conn-tri}, hence
we omit the proof. 

\begin{theo}  \label{theo:kappa-conn-quad}
Let $\kappa \in \{2,3\}$. 
Then there exists a constant $C$ such that \\
\[ W(G) \leq \frac{1}{6\kappa}n^3 + Cn^{5/2} \] 
for every $\kappa$-connected simple quadrangulation $G$ of order $n$. 
\hfill $\Box$
\end{theo}

The leading coefficients in the bounds in Theorems \ref{theo:kpa-conn-tri}
and \ref{theo:kappa-conn-quad} are optimal. This is shown by the graphs in 
Figures \ref{fig:t0}, \ref{fig:t1} and \ref{fig:t2} for $3$-connected triangulations, 
in Figures \ref{fig:t4c2}, \ref{fig:t4c3}, \ref{fig:t4c0} and \ref{fig:t4c1} for
$4$-connected triangulations, 
in Figures \ref{fig:t5c2}, \ref{fig:t5c3wi}, \ref{fig:t5c4}, \ref{fig:t5c0} and 
\ref{fig:t5c1} for $5$-connected triangulations, 
 in 
Figures \ref{fig:q0} and \ref{fig:q1} for $2$-connected quadrangulations, 
and in 
Figures \ref{fig:q3c2}, \ref{fig:q3c0} and \ref{fig:q3c1} for $3$-connected quadrangulations,

\section{Computational Results and Conjectures} \label{comput}

This section contains numerous figures and tables summarizing months of computer searches. None of this would have been possible without the help provided by Plantri, a program that generates triangulations and quadrangulation on numerous surfaces. For each category of problem (triangulations, 4-connected triangulations, 5-connected triangulations, quadrangulations and 3-connected quadrangulations) there is a table, which summarizes the largest Wiener index and  remoteness found for a given order in that category, along with ``Count", telling how many graphs attain the optimal value. Note that remoteness  in this section is \textit{ not normalized} to keep 
the calculations in the domain of integers. In other words, in the Tables we show  $(n-1){\rho (G)}$ under the
name of ``Remoteness".
Our Wiener index findings match those of \cite{CheCol2019} for triangulations. The number of isomorphism classes that our code searched matches the numbers  in \cite{BowFis1967}, \cite{BriGGMTW2005}, \cite{BriMck2005}, \cite{Lut2008}, \cite{SchSchSte2018}, verifying that
the  values that the search provides are in fact maximal. In each figure below, purple edges represent the repeating pattern and 
the black node marks a vertex which maximizes the remoteness. The computational evidence suggests that for sufficiently large order,
the maximum Wiener index is uniquely realized in every category, while remoteness is not, except for quadrangulations. 

\subsection{Computational Results for Triangulations}
Although the results of \cite{ervinarxive} made the Wiener index rows of  Table~\ref{tab:tsummary} obsolete for $n\geq 9$, we still include it to show the multiplicity
of maximizers up to $n=8$ and the remoteness results.

\FloatBarrier

\begin{figure}[htbp]
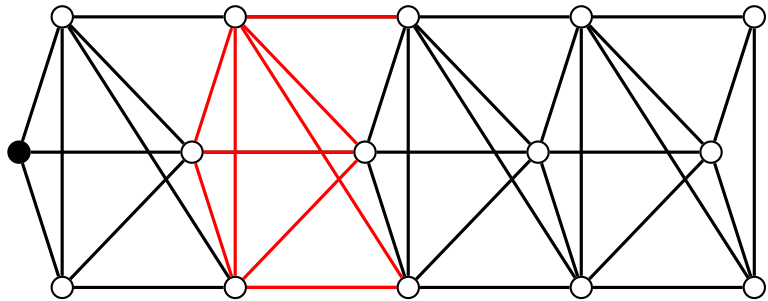
 
	\begin{center}

		\caption{A triangulation $ T_n $ on $ n = 3k $ vertices, which maximizes  the Wiener index \cite{ervinarxive} and the remoteness.}
		\label{fig:t0}
	\end{center}
\end{figure}

\begin{figure}[htbp] 
	\begin{center}
		%
		\caption{A triangulation $ T_n $ on $ n = 3k + 1 $ vertices, which maximizes  the Wiener index \cite{ervinarxive} and the remoteness.} 
		\label{fig:t1}
	\end{center}
\end{figure}

\begin{figure}[htbp] 
	\begin{center}
		%
		\caption{A triangulation $ T_n $ on $ n = 3k + 2 $ vertices, which maximizes  the Wiener index \cite{ervinarxive} and the remoteness.}
		\label{fig:t2}
	\end{center}
\end{figure}

\begin{table}
    \begin{center}
		\begin{tabular}{ |c|cc|cc| } 
			\hline
			Order & Wiener index & Count & Remoteness & Count   \\ [0.5ex] 
			\hline
			4 & 6 & 1 & 3 & 1 \\
			5 & 11 & 1 & 5 & 1  \\ 
			6 & 18 & 2 & 7 & 1  \\ 
			7 & 27 & 5 & 9 & 4  \\ 
			8 & 39 & 2 & 12 & 2  \\ 
			9 & 54 & 1 & 15 & 4  \\ 
			10 & 72 & 1 & 18 & 17  \\ 
			11 & 94 & 1 & 22 & 7  \\ 
			12 & 120 & 1 & 26 & 25 \\ 
			13 & 150 & 1 & 30 & 107 \\ 
			14 & 185 & 1 & 35 & 35  \\ 
			15 & 225 & 1 & 40 & 171  \\ 
			16 & 270 & 1 & 45 & 743  \\ 
			17 & 321 & 1 & 51 & 217  \\ 
			18 & 378 & 1 & 57 & 1199 \\ 
			\hline
		\end{tabular}
		\caption{A summary of the largest Wiener Index and remoteness among all triangulations on $ n $ vertices, and a count for how many isomorphism classes attain this value.}
		\label{tab:tsummary}
		\end{center}
	\end{table}

\begin{figure}[htbp]
	\begin{center}
		%
		\caption{A 4-connected triangulation $ T_n^4 $ on $ n = 4k+2 $ vertices, which maximizes the remoteness and is conjectured to maximize the Wiener index.}
		\label{fig:t4c2}
	\end{center}
\end{figure}

\begin{figure}[htbp]
	\begin{center}
		%
		\caption{A 4-connected triangulation $ T_n^4 $ on $ n = 4k+3 $ vertices, which maximizes the remoteness and is conjectured to maximize the Wiener index.}
		\label{fig:t4c3}
	\end{center}
\end{figure}

\begin{figure}[htbp]
	\begin{center}
		%
		\caption{A 4-connected triangulation $ T_n^4 $ on $ n = 4k $ vertices, which maximizes the remoteness and is conjectured to maximize the Wiener index.}
		\label{fig:t4c0}
	\end{center}
\end{figure}

\begin{figure}[htbp]
	\begin{center}
		%
		\caption{A 4-connected triangulation $ T_n^4 $ on $ n = 4k+1 $ vertices, which maximizes the remoteness and is conjectured to maximize the Wiener index.}
		\label{fig:t4c1}
	\end{center}
\end{figure}

\begin{table}
	\begin{center}
		\begin{tabular}{ |c|cc|cc| } 
			\hline
			Order & Wiener Index & Count & Remoteness & Count   \\ [0.5ex] 
			\hline
			
			6 & 18 & 1 & 6 & 1  \\ 
			7 & 27 & 1 & 8 & 1  \\ 
			8 & 38 & 2 & 10 & 2   \\ 
			9 & 51 & 4 & 12 & 4   \\ 
			10 & 68 & 1 & 15 & 4   \\ 
			11 & 87 & 1 & 18 & 6   \\ 
			12 & 110 & 1 & 21 & 16  \\ 
			13 & 135 & 1 & 24 & 50   \\ 
			14 & 166 & 1 & 28 & 24   \\ 
			15 & 199 & 1 & 32 & 66   \\ 
			16 & 238 & 1 & 36 & 186   \\ 
			17 & 279 & 1 & 40 & 653   \\ 
			18 & 328 & 1 & 45 & 250   \\ 
			19 & 379 & 1 & 50 & 879   \\ 
			20 & 438 & 1 & 55 & 2599   \\ 
			21 & 499 & 1 & 60 & 9429   \\ 
			22 & 570 & 1 & 66 & 3313   \\ 
			\hline
		\end{tabular}
		\caption{A summary of the largest Wiener Index and remoteness among all 4-connected triangulations on $ n $ vertices, and a count for how many isomorphism classes attain this value.}
		\label{tab:t4csummary}
		\end{center}
	\end{table}

\begin{figure}[htbp] 
	\begin{center}
		%
		\caption{A 5-connected triangulation $ T_n^5 $ on $ n = 5k+2 $ vertices, which maximizes the remoteness and is conjectured to maximize the Wiener Index.}
		\label{fig:t5c2}
	\end{center}
\end{figure}

\begin{figure}[htbp]
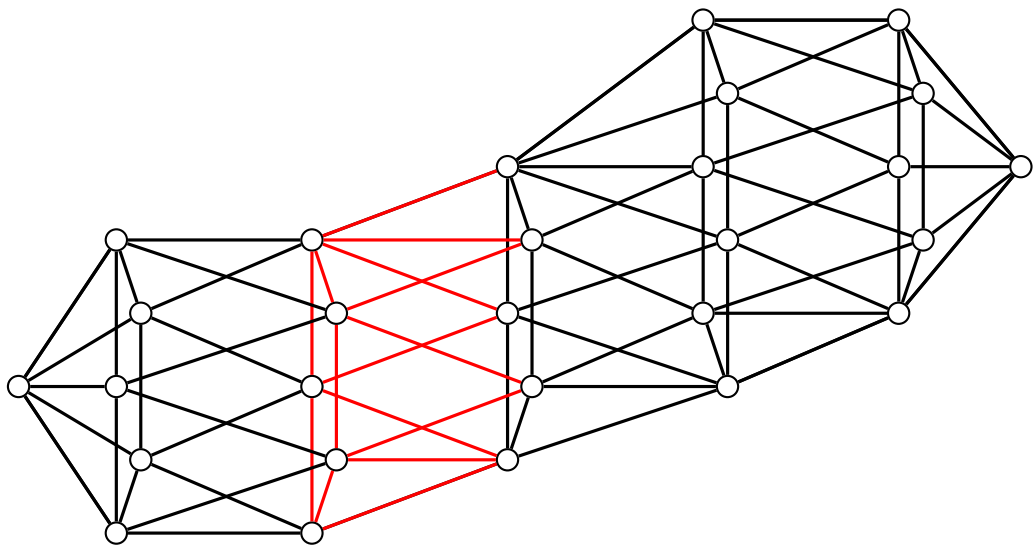
 
	\begin{center}
		%
		\caption{A 5-connected triangulation $ T_n^5 $ on $ n = 5k+3 $ vertices, which is conjectured to maximize the Wiener Index.}
		\label{fig:t5c3wi}
	\end{center}
\end{figure}

\begin{figure}[htbp]
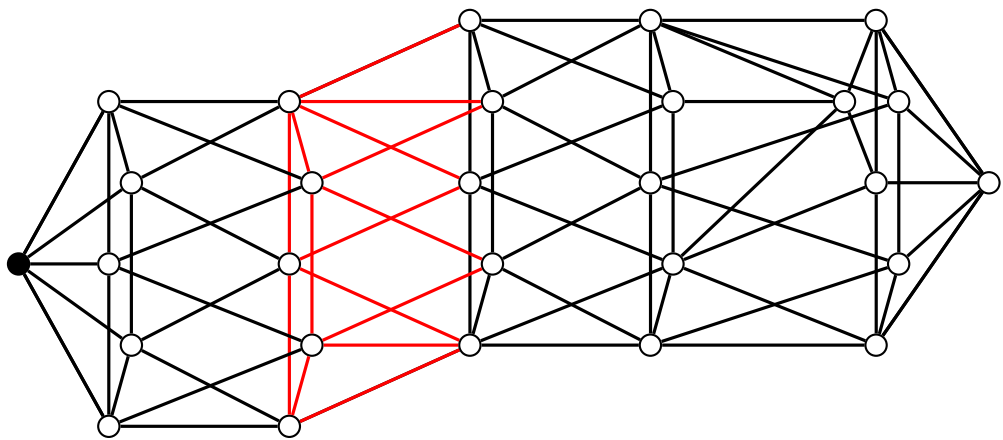
 
	\begin{center}
		%
		\caption{A 5-connected triangulation $ T_n^5 $ on $ n = 5k+3 $ vertices which maximizes the remoteness.}
		\label{fig:t5c3remoteness}
	\end{center}
\end{figure}

\begin{figure}[htbp] 
	\begin{center}
		%
		\caption{A 5-connected triangulation $ T_n^5 $ on $ n = 5k+4 $ vertices, which maximizes the remoteness and is conjectured to maximize the Wiener Index.}
		\label{fig:t5c4}
	\end{center}
\end{figure}

\begin{figure}[htbp] 
	\begin{center}
		%
		\caption{A 5-connected triangulation $ T_n^5 $ on $ n = 5k $ vertices, which maximizes the remoteness and is conjectured to maximize the Wiener Index.}
		\label{fig:t5c0}
	\end{center}
\end{figure}

\begin{figure}[htbp] 
	\begin{center}
		%
		\caption{A 5-connected triangulation $ T_n^5 $ on $ n = 5k+1 $ vertices which maximizes the remoteness and is conjectured to maximize the Wiener index.}
		\label{fig:t5c1}
	\end{center}
\end{figure}

\begin{table}
	\begin{center}
		\begin{tabular}{ |c|cc|cc| } 
			\hline
			Order & Wiener index & Count & Remoteness & Count   \\ [0.5ex] 
			\hline

			12 & 108 & 1 & 18  & 1  \\ 
			13 & --- & 0 & ---  & 0  \\ 
			14 & 159 & 1 & 23 & 1  \\ 
			15 & 189 & 1 & 26 & 1  \\ 
			16 & 222 & 2 & 29 & 1  \\ 
			17 & 259 & 1 & 34 & 1  \\ 
			18 & 300 & 1 & 37 & 1  \\ 
			19 & 342 & 1 & 41 & 2  \\ 
			20 & 391 & 1 & 45 & 4  \\ 
			21 & 444 & 1 & 49 & 9  \\ 
			22 & 500 & 2 & 55 & 4  \\ 
			23 & 560 & 1 & 59 & 11 \\ 
			24 & 630 & 1 & 64 & 36 \\ 
			25 & 702 & 1 & 69 & 66  \\ 
			26 & 780 & 1 & 74 & 193  \\ 
			27 & 867 & 1 & 81 & 39  \\ 
			28 & 955 & 1 & 86 & 240  \\ 
			29 & 1053 & 1 & 92 & 805  \\ 
			30 & 1156 & 1 & 98 & 1470 \\ 
			31 & 1265 & 1 & 104 & 4327   \\ 
			32 & 1384 & 1 & 112 & 763   \\ 
			
			\hline
		\end{tabular}
		\caption{A summary of the largest Wiener Index and remoteness among all 5-connected triangulations on $ n $ vertices, and a count for how many isomorphism classes attain this value.}
		\label{tab:t5csummary}
		\end{center}
	\end{table}

\FloatBarrier
\subsection{Computational Results for Quadrangulations}
\FloatBarrier
\begin{figure}[htbp]
	\begin{center}
		%
		\caption{A quadrangulation $ Q_n $ on $ n = 2k $ vertices, which maximizes the remoteness and is conjectured to maximize the Wiener index.} 
		\label{fig:q0}
	\end{center}
	
\end{figure}

\begin{figure}[htbp]
	\begin{center}
		%
		\caption{A quadrangulation $ Q_n $ on $ n = 2k+1 $ vertices, which maximizes the remoteness and is conjectured to maximize the Wiener index.} 
		\label{fig:q1}
	\end{center}
	
\end{figure}

\begin{table}
    \begin{center}

		\begin{tabular}{ |c|cc|cc| } 
			\hline
			Order & Wiener Index & Count & Remoteness & Count   \\ [0.5ex] 
			\hline
			4 & 8 & 1 & 4 & 1  \\ 
			5 & 14 & 1 & 6 & 1  \\ 
			6 & 23 & 1 & 9 & 1  \\ 
			7 & 34 & 2 & 12 & 1  \\ 
			8 & 50 & 1 & 16 & 1  \\ 
			9 & 68 & 1 & 20 & 1  \\ 
			10 & 93 & 1 & 25 & 1  \\ 
			11 & 120 & 1 & 30 & 1  \\ 
			12 & 156 & 1 & 36 & 1  \\ 
			13 & 194 & 1 & 42 & 1  \\ 
			14 & 243 & 1 & 49 & 1  \\ 
			15 & 294 & 1 & 56 & 1  \\ 
			16 & 358 & 1 & 64 & 1  \\ 
			17 & 424 & 1 & 72 & 1  \\ 
			18 & 505 & 1 & 81 & 1  \\ 
			19 & 588 & 1 & 90 & 1  \\ 
			20 & 688 & 1 & 100 & 1 \\ 
			\hline
		\end{tabular}
		\caption{A summary of the largest Wiener Index and remoteness among all quadrangulations on $ n $ vertices, and a count for how many isomorphism classes attain this value.}
        \label{tab:qsummary}
		\end{center}
	\end{table}

\begin{figure}[htbp] 
	\begin{center}
		%
		\caption{A 3-connected quadrangulation $ Q_n^3 $ on $ n = 3k+14 $ vertices, which maximizes the remoteness and is conjectured to maximize the Wiener index.}
		\label{fig:q3c2} 
	\end{center}
\end{figure}

\begin{figure}[htbp] 
	\begin{center}
		%
		\caption{A 3-connected quadrangulation $ Q_n^3 $ on $ n = 3k+15 $ vertices, which maximizes the remoteness and is conjectured to maximize the Wiener index.}
		\label{fig:q3c0}
	\end{center}
\end{figure}

\begin{figure}[htbp] 
	\begin{center}
		%
		\caption{A 3-connected quadrangulation $ Q_n^3 $ on $ n = 3k+16 $ vertices which maximizes the remoteness and is conjectured to maximize the Wiener index.}
		\label{fig:q3c1} 
	\end{center}
\end{figure}

\begin{table}
	\begin{center}
		\begin{tabular}{ |c|cc|cc| } 
			\hline
			Order & Wiener index & Count & Remoteness & Count \\ [0.5ex] 
			\hline
			
			8 & 48 & 1 & 12  & 1 \\ 
			9 & --- & 0 & --- & 0\\ 
			10 & 83 & 1 & 17 & 1  \\ 
			11 & 106 & 1 & 22 & 1  \\ 
			12 & 136 & 1 & 24 & 2  \\ 
			13 & 164 & 1 & 29 & 2  \\ 
			14 & 201 & 1 & 35 & 2  \\ 
			15 & 240 & 1 & 38 & 6  \\ 
			16 & 288 & 2 & 44 & 7  \\ 
			17 & 344 & 1 & 51 & 5  \\ 
			18 & 401 & 1 & 55 & 26 \\ 
			19 & 468 & 1 & 62 & 33  \\ 
			20 & 544 & 1 & 70 & 22  \\ 
			21 & 622 & 1 & 75 & 136  \\ 
			22 & 711 & 1 & 83 & 172  \\ 
			23 & 810 & 1 & 92 & 97  \\ 
			24 & 912 & 1 & 98 & 729  \\ 
			25 & 1026 & 1 & 107 & 923  \\ 
			26 & 1151 & 1 & 117 & 505  \\ 
			27 & 1280 & 1 & 124 & 3930  \\ 
			28 & 1422 & 1 & 134 & 4959  \\

			\hline
		\end{tabular}
		\caption{A summary of the largest Wiener Index and remoteness among all 3-connected quadrangulations on $ n $ vertices, and a count for how many isomorphism classes attain this value.}
        \label{tab:q3csummary}
		\end{center}
\end{table}

\FloatBarrier

\end{document}